\documentclass[a4paper,11pt,reqno]{amsart}
\usepackage{amsmath}
\usepackage[T1]{fontenc}
\usepackage{amssymb}
\usepackage{amsthm}
\usepackage{color}
\usepackage[lmargin=2.5 cm,rmargin=2.5 cm,tmargin=3.5cm,bmargin=2.5cm,
paper=a4paper]{geometry}

\newcommand{\Ab}{\mathbf A}

\newcommand{\Bb}{\mathbf B}

\newcommand{\R}{\mathbb R}

\DeclareMathOperator{\curl}{curl}
\DeclareMathOperator{\tr}{tr} 
\DeclareMathOperator{\supp}{supp} 

\newtheorem{thm}{Theorem}[section]

\newtheorem{lem}[thm]{Lemma}

\theoremstyle{remark}

\numberwithin{equation}{section}

\title[Spectral asymptotics]{Spectral asymptotics for magnetic Schr\"odinger operators in domains with corners}

\author[A.~Kachmar]{Ayman Kachmar$^{\rm a, b}$}
\author[A. Khochman]{Abdallah Khochman$^{\mathrm a}$}

\address{$^{\mathrm a}$ Lebanese University, Department of Mathematics, Hadath, Lebanon}
\address{
$^{\mathrm b}$ Lebanese International University, School of Arts and Sciences, Beirut, Lebanon}
\email[A. Kachmar]{ayman.kashmar@liu.edu.lb}
\email[A. Khochman]{akhochman@hotmail.fr}
\date{\today\\
{\it Mathematics Subject Class Classification.} 81Q10, 35J10, 35P15}

\begin{document}
\maketitle

\begin{abstract}
This paper is on magnetic Schrodinger operators in two dimensional
domains with corners.  Semiclassical formulas are obtained for the
sum and number of eigenvalues. The obtained results extend former
formulas for smooth domains in \cite{Fr, FK} to  piecewise smooth
domains.
\end{abstract}

\section{Introduction}

The spectral analysis of  magnetic Schr\"odinger operators in
domains with boundary has been the subject of many research papers
in the last two decades. Apart from the mathematical interest behind
the study of their spectra, magnetic Schr\"odinger operators in
interior/exterior domains with various boundary conditions arise in
several models of condensed matter physics, as superconductivity
\cite{HeMo, FH-book, Kach1, Kach6}, liquid crystals \cite{HP, Pa}
and Fermi gases \cite{CFFH}.

The present paper is devoted to the study of magnetic Schr\"odinger
operators in domains with corners (piecewise smooth domains). The
presence of corners in the domain has a strong effect on the
spectrum of the operator. In particular, it is shown in  \cite{Bon,
Ja} that the presence of corners decreases the value of the ground
state energy of the operator compared with the case of smooth
domains. Discussion of this effect in the framework of
superconductivity is given in \cite{BoFo}.

We give in a simple particular case, a brief presentation of the
semiclassical results proved in this paper. Suppose for simplicity
that $\Omega$ is a simply connected bounded domain in $\R^2$ and
that $\Ab$ is a vector field such that $b=\curl \Ab$ is constant.
Let $P_{h,\Omega}=-(h\nabla-i\Ab)^2$ be the magnetic Laplacian in
$L^2(\Omega)$ with magnetic Neumann boundary condition, $N(\lambda
h)$ and $E(\lambda h)$ be the number and sum of negative eigenvalues
of $P_{h,\Omega}-\lambda Id$. If the boundary of $\Omega$ is smooth,
it is proved in \cite{FK, Fr} that, as $h\to0$,
\begin{equation} \label{eq-intr}N(\lambda h)\sim
c_1(b,\lambda)|\partial\Omega|h^{-1/2}\,,\quad E(\lambda h)\sim
c_2(b,\lambda) |\partial\Omega|h^{1/2}\,,\end{equation} where $c_1$
and $c_2$ are two explicit constants. The formula for the number
$N(\lambda h)$ is valid for all $\lambda<b$, while that for the
energy $E(\lambda h)$ is valid for all $\lambda\leq b$.

In Theorem~\ref{thm:FK;Fr} of this paper, we prove that
\eqref{eq-intr} is still holding true when the domain $\Omega$ is
only {\it piecewise} smooth. The key to prove this extension is a
rough estimate on  the number of eigenvalues of  a Schr\"odinger
operator with constant magnetic field in a sectorial domain (see
Lemma~\ref{thm-disspec}).

The result of Theorem~\ref{thm:FK;Fr} shows that corners {\it only}
affect the low-lying eigenvalues of the operator. As we will see in
Theorem~\ref{thm-nb}, corners create {\it a few} additional
eigenvalues compared with smooth domains.

As a consequence, we may say that in the semi-classical limit and
the regime considered in this paper, the number of eigenvalues of
magnetic Schr\"odinger operators for smooth and non-smooth domains
are {\it asymptotically} the same.

The paper is organized in the following way. In
Section~\ref{sec-prelim}, we collect some key results that will be
used throughout the paper.  Section~\ref{sec:nb:corn} is about the
number of low-lying eigenvalues of the operator in a domain with
corners (here the regularity assumption on the domain is precisely
stated). Finally, Section~\ref{sec:extension} contains the
semi-classical analysis that extends \eqref{eq-intr} to domains with
corners.

Throughout the paper, the notation $N(\lambda, P)$ will be often
used to denote the number of eigenvalues (counting multiplicities)
of the operator $P$ that are below $\lambda$.

\section{Preliminaries}\label{sec-prelim}

\subsection{Variational principles}

Let $H$ be a self-adjoint operator in a Hilbert space $\mathcal H$
(of domain $D(H)$) such that
$$({\rm H})\quad
\left\{
\begin{array}{l}
\inf\,\sigma_{\rm ess}(H)\geq 0\\
H\mathbf 1_{(-\infty,0)}(H) \text{ is trace class}\,.
\end{array}
\right.$$

We shall need the following two simple variational principles
concerning the operator $H$, which are frequently used in
\cite{LSY,ES}.

\begin{lem}\label{lem-VP1}
Let $\gamma$ be a bounded operator such that $0\leq\gamma\leq 1$ (in
the sense of quadratic forms) and the operator $H\,\gamma$ is trace
class. Then it holds that,
$$\tr\left(H\mathbf 1_{(-\infty,0)}(H)\right)\leq \tr(H\,\gamma)\,.
$$
\end{lem}

\begin{lem}\label{lem-VP2}
Assume that the operator $H$ satisfies the hypothesis {\rm (H)}.
Then it holds that,
$$\tr\left(H\mathbf 1_{(-\infty,0)}(H)\right)
=\inf\sum_{j=1}^N\langle f_j\,,\, H \,f_j\rangle\,,$$ where the
infimum is taken over all orthonormal families
$\{f_1,f_2,\ldots,f_N\}\subset D(H)$ and $N\geq1$.
\end{lem}

\subsection{A family of one-dimensional differential operators}
\label{Sub-Sec-M0-1}

Let us recall the main results obtained in \cite{DaHe,HeMo}
concerning the family of harmonic oscillators with Neumann boundary
condition. Given $\xi\in\mathbb R$, we define the quadratic form,
\begin{equation}\label{k2-qfk1}
B^1(\mathbb R_+)\ni u\mapsto q[\xi](u)=\int_{\mathbb
R_+}|u'(t)|^2+|(t-\xi)u(t)|^2 d  t,
\end{equation}
where, for a positive integer $k\in\mathbb N$ and a given interval
$I\subseteq\mathbb R$, the space $B^k(I)$ is defined by~:
\begin{equation}\label{Bk-sp}
B^k(I)=\{u\in H^k(I);\quad t^ju(t)\in L^2(I),\quad \forall
j=1,\ldots, k\}.
\end{equation}
Since the quadratic form (\ref{k2-qfk1}) is closed and symmetric it
defines a unique self-adjoint operator  $\mathcal L[\xi]$. This
operator has domain,
$$D(\mathcal L[\xi])=\{u\in B^2(\mathbb R_+);\quad
u'(0)=0\},$$ and is the realization of the differential operator,
\begin{equation}\label{L-dop}
\mathcal L[\xi]=-\partial_t^2+(t-\xi)^2,
\end{equation}
on the given domain. We denote by $\{\mu_j(\xi)\}_{j=1}^{+\infty}$
the increasing sequence of eigenvalues of $\mathcal L[\xi]$, which
are all simple. By the min-max principle, we have,
\begin{equation}\label{l1-ga,xi}
\mu_1(\xi)=\inf_{u\in B^1(\mathbb
R_+),u\not=0}\frac{q[\xi](u)}{\|u\|^2_{L^2(\mathbb R_+)}}.
\end{equation}
It follows from analytic perturbation theory (see \cite{DaHe}) that
the functions
$$\mathbb R\ni\xi\mapsto\mu_j(\xi)$$ are analytic.
Furthermore, $\mu_1(0)=1$ and $|\mu_1(\xi)-1|$
decays like $\exp(-\xi^2)$ as $\xi\to+\infty$ (see \cite{BoHe}) thus
yielding that
$$\int_{0}^\infty(\mu_1(\xi)-1)\, d
\xi=-\int_{\R}[\mu_1(\xi)-1]_-\, d \xi\quad\text{  is finite.}$$
We define the constant~:
\begin{equation}\label{Th-gam}
\Theta_0=\inf_{\xi\in\mathbb R}\mu_1(\xi).
\end{equation}
In \cite{DaHe}, it is proved that $\Theta_0=\mu_1(\xi_0)$, that
$\xi_0$ is the unique value at which the minimum $\Theta_0$ is
attained and that $\mu''(\xi_0)>0$.

An important consequence of standard Sturm-Liouville theory is
recalled below (c.f. \cite[Lemma~2.1]{Fr}).
\begin{lem}\label{RuFr-Lemm-kach}
The second eigenvalue of $\mathcal L[\xi]$ satisfies,
$$\inf_{\xi\in\mathbb R}\mu_2(\xi)>1\,.$$
\end{lem}

Notice that part of this conclusion is a consequence of the analysis
of Dauge-Helffer \cite{DaHe}, who show that the infimum of
$\mu_2(\xi)$ is attained for a unique $\xi_2\in\mathbb R$.

\subsection{Rough energy bound for the cylinder}\label{sec-torus}

Let us consider the operator
$$P_\Omega=-(h\nabla-ib\Ab_0)^2\quad {\rm in}~L^2(\Omega)\,,$$
with
$$\Omega=[0,S]\,\times\,(0,h^{1/2}T)\,.$$ Functions in the domain
of  $P_{h,\Omega}$ satisfy Neumann condition at $t=0$, periodic
conditions at $s\in\{0,S\}$ and Dirichlet condition at $t=h^{1/2}T$.
We assume that the vector field $\Ab_0$ is given by
$$\Ab_0(x_1,x_2)=(-x_2,0)\,.$$
 In this particular case, the operator has compact resolvent, hence
the spectrum consists of an increasing sequence of eigenvalues
$(e_j)_{j\geq1}$ converging to $+\infty$. Note that the terms of the sequence $(e_j)$ are listed with  multiplicities counted.
Given
$\lambda\in\R$, the energy
\begin{equation}\label{energy-torus}
\mathcal
E(\lambda,b,S,T)=\sum_j\left[hb(1+\lambda)-e_j\right]_+\end{equation}
is finite. The number  of eigenvalues below $(1+\lambda)hb$ is also
of particular interest\,:
\begin{equation}\label{nb-torus}
\mathcal N(\lambda,b,S,T)={\rm Card}\,\{j~:~e_j\leq hb(1+\lambda)\}\,.
\end{equation}

In the lemma below, we state estimates of the energy $\mathcal
E(\lambda,b,S,T)$ and the number $\mathcal N(\lambda,b,S,T)$.

\begin{lem}\label{roughestimate}
There exist positive constants $T_0$ and $\lambda_0$ such that, for
all $S>0$, $b>0$, $T\geq \sqrt{b}\,T_0$ and
$\lambda\in(0,\lambda_0]$, we have,
\begin{equation}\label{eq-cyl:en}
\mathcal E(\lambda,b,S,T)\leq  (1+\lambda)hb\left(\frac{ST}{2\pi\sqrt{h}}+1\right)\,,\end{equation}
\begin{equation}\label{eq-cyl:nb}
\mathcal N(\lambda,b,S,T)\leq  \frac{ST}{2\pi\sqrt{h}}+1\,.
\end{equation}
\end{lem}

The lemma is proved in \cite{FK}, although the estimate
\eqref{eq-cyl:nb} is not stated explicitly. Actually,
\eqref{eq-cyl:en} is proved as follows. First \eqref{eq-cyl:nb} is
established by separation of variables and the variational min-max
principle then the energy $\mathcal E(\lambda,b,S,T)$ is easily
estimated as $\mathcal N(\lambda,b,S,T)\times(1+\lambda)bh$, from
which \eqref{eq-cyl:en} follows.

\subsection{Rough bounds for the sector}

Let $\alpha\in(0,2\pi)$, $R>0$, $h>0$, $b>0$ and
\begin{equation}\label{Omega-alpha'}
\Omega_{R,h,\alpha}=\{\,(r\cos\theta,r\sin\theta)\in\R^2~:~0<\theta<\alpha\,,\quad 0\leq r<h^{1/2}R\,\}\,.
\end{equation}

Consider the self-adjoint operator
$P_{h,\alpha}=-(h\nabla-ib\Ab_0)^2$ in $L^2(\Omega_{R,h,\alpha})$. Functions in the domain of $P_{h,\alpha}$ satisfy Dirichlet condition on the boundary $r=R$ and Neumann condition
 $$\nu\cdot(h\nabla-ib\Ab_0)u=0$$
 on the boundary defined by $\theta=0$ or $\theta=\alpha$. Here $\nu$ is the unit outward normal vector on the boundary $\partial\Omega_{R,h,\alpha}$ and $\Ab_0$ is the magnetic potential introduced in \eqref{eq:mp}.

The operator $P_{h,\alpha}$ has compact resolvent and its spectrum is discrete and consists of isolated eigenvalues $(e_j)$ counted with multiplicities. Let $\lambda\in\R$ and define,
\begin{equation}\label{eq-sec:en}
\mathcal E_{\rm corn}(\alpha,b,R,\lambda)=\sum_j\left[hb(1+\lambda)-e_j\right]_+\,,
\end{equation}
and
\begin{equation}\label{eq-sec:nb}
\mathcal N_{\rm corn}(\alpha,b,R,\lambda)={\rm Card}\,\{j~:~e_j\leq hb(1+\lambda)\,\}\,.
\end{equation}

We give rough estimates of $\mathcal E_{\rm corn}(\alpha,b,R,\lambda)$ and $\mathcal N_{\rm corn}(\alpha,b,R,\lambda)$ in the next lemma.

\begin{lem}\label{thm-disspec}
Given $b>0$ and $\alpha\in[0,2\pi)$, there exist positive constants
$C$, $h_0$, $R_0$ and $\lambda_1$ such that, for all $R\geq R_0$,
$h\in(0,h_0]$,  and $\lambda\in(-\infty,\lambda_1]$, we have,
\begin{equation}\label{eq-sec:nb'}
\mathcal N_{\rm corn}(\alpha,b,R,\lambda)\leq  C\left(R^2+1\right)\,,\end{equation}
and
\begin{equation}\label{eq-sec:en'}
\mathcal E_{\rm corn}(\alpha,b,R,\lambda)\leq  Chb\left(R^2+1\right)\,.
\end{equation}
\end{lem}
\begin{proof}
Observe that the upper bound \eqref{eq-sec:en'} on the energy
$\mathcal E_{\rm corn}(\alpha,b,R,\lambda)$ follows immediately from
the definition of $\mathcal E_{\rm corn}(\alpha,b,R,\lambda)$ and
the upper bound in \eqref{eq-sec:nb'}. Therefore, we establish the
bound in \eqref{eq-sec:nb'}. The method that will be used is
introduced in \cite{CdV} and based on a decomposition of the
operator via a partition of unity and on the min-max variational
principle.

To simplify the notation, we denote by $\Omega$ and $P$ the set $\Omega_{R,h,\alpha}$ introduced in \eqref{Omega-alpha'} and the operator $P_{h,\alpha}$ respectively.

Let $A_0$ be a positive real number whose choice will be specified
later. Suppose $R>A_0$. We cover $\Omega$ by two sets,
\begin{multline*}
U_1(h)=\{x\in\Omega~:~|x|< A_0h^{1/2}\}\,,\quad
U_2(h)=\{x\in\Omega~:~ \frac{A_0 h^{1/2}}2< |x|< Rh^{1/2}\}\,.
\end{multline*}
Consider a partition of unity of
$\Omega$,
$$\sum_{j=1}^2\chi_j^2(x)=1\quad{\rm and}\quad \sum_{j=1}^2|\nabla\chi_j(x)|^2\leq \frac{C_1}{A_0^2h}~{\rm
in}~\Omega\,,$$ where $C_1>0$ is a universal constant   and ${\rm supp}\,\chi_j\subset U_j(h^{1/2})$.

Let $u\in H^1(\Omega)$. The following decomposition formula
holds true,
$$\int_{\Omega}|(\nabla-ib\Ab_0)u|^2\,dx=
\sum_{j=1}^2\int_{\Omega}\left(|(\nabla-ib\Ab_0)\chi_{j}u|^2
-h^2|\nabla\chi_{j}|^2\,|u|^2\right)dx\,.$$ This decomposition
yields the inequality stated  below by using the upper bound on
$|\nabla\chi_{j}|$~:
$$\int_{\Omega}|(\nabla-ib\Ab_0)u|^2\,dx\geq
\sum_{j=1}^2\int_{\Omega}\left(|(\nabla-ib\Ab_0)\chi_{j}u|^2
-C_1A_0^{-2}h\,|u|^2\right)dx\,.$$
By using the method in \cite{CdV}
and the variational min-max principle, we get that,
\begin{equation}\label{eq-ub-number}
\mathcal N_{\rm corn}(\alpha,b,R,\lambda)\leq N(\Lambda h,P_{U_1(h)})+N(\Lambda h,P_{U_2(h)})\,,
\end{equation}
where $\Lambda=(1+\lambda)b+C_1A_0^{-2}$,
$P_{U_1(h)}=-(h\nabla-ib\Ab_0)^2$ is the operator in $L^2(U_1(h))$
with Dirichlet condition on $r=h^{1/2}$ and Neumann condition on the
other parts of the boundary of $U_1(h)$, and
$P_{U_2(h)}=-(h\nabla-ib\Ab_0)^2$ is the operator in $L^2(U_2(h))$
with Neumann boundary condition on $\theta=0$ and $\theta=\alpha$
and Dirichlet condition on the other parts of the boundary. Notice
that we use polar coordinates $(r,\theta)$ in the definition of the
domains.

Let $\lambda_0$ be as in the statement of Lemma~\ref{roughestimate}
and select $\lambda_1\in(0,\lambda_0)$. In light of
\eqref{eq-ub-number}, Theorem~\ref{thm-disspec} will be proved once
the  statement that follows below is shown to be true:
\begin{align}
\exists~C>0,&~N(\Lambda h,P_{U_1(h)})\leq C{\rm ~and}\label{eq-ub-number**}\\
&~N(\Lambda h,P_{U_2(h)})\leq C\left(R^2+1\right)\,.
\label{eq-ub-number*}
\end{align}

\paragraph{\it Proof of \eqref{eq-ub-number**}}
Applying the re-scaling $y=h^{-1/2}x$ changes the domain $U_1(h)$ to
$U_1(A_0)$ and the operator $P_{U_1(h)}$ to the operator in
$L^2(U_1(A_0))$,
$$P_{U_1(A_0)}=-(\nabla-ib\Ab_0)^2\,.$$
More precisely, a number $e$ is in the spectrum of $P_{U_1(h)}$ if
and only if $e/ h$ is in the spectrum of $P_{U_1(A_0)}$. Thus, the
term we wish to estimate is given as follows,
$$N(\Lambda h,P_{U_1(h)})=N\big(\Lambda,P_{U_1(A_0)}\big)\,.$$
Since the operator $P_{U_1(1)}$ has compact resolvent and the number
$\Lambda$ is bounded independently of $h\in(0,1]$, it follows that
$N\big(\Lambda,P_{U_1(1)}\big)$ is bounded as  $h$ varies in $(0,1]$
too. This establishes the upper bound in \eqref{eq-ub-number**}.

\paragraph{\it Proof of \eqref{eq-ub-number*}}
As we shall see, the upper bound in \eqref{eq-ub-number*} follows
mainly from Lemma~\ref{roughestimate}. Let $\epsilon$ be a positive
real number such that,
$$0<\epsilon<\min\left(\frac\alpha8,\frac\pi4-\frac\alpha8\right)\,.$$ Recall the definition of the set
$$U_2(h)=\{x\in\Omega~:~
\frac{A_0h^{1/2}}2< |x|< Rh^{1/2}\}\,.
$$
We cover $U_2(h)$ by four sets,
\begin{align*}
&V_1(h)=\{(r\cos\theta,r\sin\theta)\in \overline{U_2(h)}~:~0\leq \theta< \frac\alpha4\,\}\,,\\
&V_2(h)=\{(r\cos\theta,r\sin\theta)\in \overline{U_2(h)}~:~\frac\alpha4-\epsilon< \theta<\frac{\alpha}2+\epsilon\,\}\,,\\
&V_3(h)=\{(r\cos\theta,r\sin\theta)\in \overline{U_2(h)}~:~\frac\alpha2-\epsilon< \theta<\frac{3\alpha}4+\epsilon\,\}\\
&V_4(h)=\{(r\cos\theta,r\sin\theta)\in \overline{U_2(h)}~:~\frac{3\alpha}4< \theta\leq\alpha\,\}\,.
\end{align*}
Define four operators $P_{V_1(h)}$, $P_{V_2(h)}$, $P_{V_3(h)}$ and
$P_{V_4(h)}$ in $L^2(V_1(h))$, $L^2(V_2(h))$, $L^2(V_3(h))$ and
$L^2(V_4(h))$ respectively. The four operators are self-adjoint
realizations of the differential operator $-(h\nabla-ib\Ab_0)^2$.
Functions in the domain of $P_{V_1(h)}$ satisfy Neumann condition on
$\theta=0$ and Dirichlet condition on the other parts of the
boundary. Similarly, functions in the domain of $P_{V_4(h)}$ satisfy
Neumann condition on $\theta=\alpha$ and Dirichlet condition on the
other parts of the boundary. Functions in the domains of
$P_{V_2(h)}$ and $P_{V_3(h)}$ satisfy Dirichelt boundary condition.
Thus, the bottoms of the spectra of the operators $P_{V_2(h)}$ and
$P_{V_3(h)}$ are larger or equal to $bh$.

Notice that the operators $P_{V_1(h)}$ and $P_{V_4(h)}$ are unitary equivalent and hence have same spectra. Also, the operators $P_{V_2(h)}$ and $P_{V_3(h)}$ are unitary equivalent and have same spectra.

We apply an argument similar to the one we did to obtain
\eqref{eq-ub-number}. By introducing a partition of unity supported
in $V_1(h)$, $V_2(h)$, ${V_3(h)}$ and  $V_4(h)$, using the IMS
decomposition formula and the variational min-max principle, we get
a constant $C_2>0$ that is allowed to depend on $\epsilon$ but not
on $h$ and such that,
\begin{align}
N(\Lambda h,P_{U_2(h)})&\leq  \sum_{j=1}^4 N(\widetilde \Lambda h,P_{V_j(h)})\nonumber\\
&=2\sum_{j=1}^2 N(\widetilde \Lambda
h,P_{V_j(h)})\,,\label{eq-up-numb-''}\end{align} with
$\widetilde\Lambda=\Lambda+C_1A_0^{-2}+C_2h$.

Recall the number $\mathcal N(\lambda,b,S,T)$ introduced in \eqref{nb-torus}. This number counts the eigenvalues of the operator $P_U$ in the cylinder, where $U=[0,S]\times (0,h^{1/2}T)$.

Since the angle $\alpha/4$ is acute and   Dirichlet condition is imposed
in the form domain of $P_{V_1(h)}$,
functions in the form domain of $P_{V_1(h)}$ can be extended by zero to functions defined in the cylinder,
$$U=[\,0, Rh^{1/2}\,]\,\times\, (\,0, R h^{1/2}\,)\,.$$
The extended functions are in the form domain of the operator $P_U$.
Thus, by the variational min-max principle, it is easy to see that,
$$N(\Lambda h,P_{V_1(h)})\leq  \mathcal N\big(\,\widetilde \lambda\,,b,Rh^{1/2}, R\,\big)\,,$$
with $\widetilde\lambda =\lambda+2b^{-1}C_1A_0^{-2}+C_2b^{-1}h$.
Similarly, we have,
$$N(\Lambda h,P_{V_2(h)})\leq  \mathcal N\big(\,\widetilde \lambda\,,b,Rh^{1/2}, R\,\big)\,.$$
 Thus we get from \eqref{eq-up-numb-''},
\begin{equation}\label{eq-up-numb-3'}
N(\Lambda h,P_{U_2(h)})\leq 2\mathcal N\big(\,\widetilde \lambda\,,b,Rh^{1/2}, R\,\big)\,.
\end{equation} Select $h_0$ sufficiently small and $A_0$ sufficiently large such that
$$2b^{-1}C_1A_0^{-2}+C_2b^{-1}h_0<\lambda_0-\lambda_1\,.$$
In this way, we get  for all $\lambda\in(0,\lambda_1)$ and
$h\in(0,h_0]$, $\widetilde\lambda\in(0,\lambda_0]$. Consequently, it
follows from Lemma~\ref{roughestimate} that $\mathcal
N\big(\,\widetilde \lambda\,,b,Rh^{1/2}, R\,\big)$ is bounded by a
constant times $(R^2+1)$, and thereby get the  upper bound in
\eqref{eq-ub-number*}.
\end{proof}

\section{Asymptotic number of low-lying eigenvalues near
corners}\label{sec:nb:corn}

\subsection{The operator in an infinite sector}\label{sec:alpha}
Let $\alpha\in(0,2\pi)$ and
\begin{equation}\label{Omega-alpha}
\Omega_\alpha=\{\,(r\cos\theta,r\sin\theta)\in\R^2~:~r\geq0\,,\quad 0<\theta<\alpha\,\}\,.
\end{equation}
Define the magnetic potential, \begin{equation}\label{eq:mp}\forall
~(x_1,x_2)\in\R^2\,,\quad \Ab_0(x_1,x_2)=(-x_2,0)\,,\end{equation}
whose curl is constant and equal to $1$.

Consider the self-adjoint operator
$P_{\Omega_\alpha}=-(\nabla-i\Ab_0)^2$ in $L^2(\Omega_\alpha)$ whose
domain is:
$$D(P_{\Omega_\alpha})=\{\,u\in
H^1(\Omega_\alpha)~:~P_{\Omega_\alpha}u\in
L^2(\Omega_\alpha)\,,~\nu\cdot(\nabla-i\Ab_0)u=0~{\rm
on}~\partial\Omega_\alpha\setminus\{0\}\,\}\,,$$ where $\nu$ is the unit outward normal vector of $\partial\Omega_\alpha$.

It is proved in \cite{Bon} that the bottom of the essential spectrum
of $P_{\Omega_\alpha}$ is a universal constant
$\Theta_0\in(\frac12,1)$. Furthermore, when
$\alpha\in(0,\frac\pi2]$, the operator $P_{\Omega_\alpha}$ has
discrete spectrum below $\Theta_0$.

Many questions connected with the spectrum of the operator
$P_{\Omega_\alpha}$ are left open. Among these questions are the
following:
\begin{itemize}
\item Is the spectrum of $P_{\Omega_\alpha}$ below $\Theta_0$
finite ?
\item When $\alpha=\pi$, the operator $P_{\Omega_\pi}$ is the
half-plane and its spectrum is purely essential and consists of
the interval $[\Theta_0,\infty)$. Is $[\Theta_0,\infty)$ the
essential spectrum of $P_{\Omega_\alpha}$ for any $\alpha$ ?
\item It is conjectured in \cite{Bon} that
the ground state energy (bottom of the spectrum) of
$P_{\Omega_\alpha}$ is constant and equal to $\Theta_0$ when
$\alpha\in[\pi,2\pi)$. This conjecture has not been proved or
disproved yet (see \cite{Bon}).
\end{itemize}

\subsection{Assumptions on the domain}\label{sec:omega}
We describe precisely the regularity properties of the domain
$\Omega$. The assumptions will be the same as those made in
\cite{Bon, BoFo}.

In this and the subsequent sections,  $\Omega$ is an  open and connected set in
$\R^2$ whose boundary is compact and consists of  a curvilinear
polygon of class $C^3$. By saying that the boundary $\Gamma$ of
$\Omega$ is a curvilinear polygon (of class $C^3$) we mean the
following (see {\cite[p.34--42]{Gr}). For every $x\in\Gamma$, there
exists a neighborhood $V$ of $x$ in $\mathbb R^2$ and a mapping
$\psi$ from $V$ to $\mathbb R^2$ such that
\begin{enumerate}
\item $\psi$ is injective,
\item $\psi$ together with $\psi^{-1}$ (defined on $\psi(V)$) belongs to the class $C^{3}$,
\item $\Omega\cap V$ is either
$\{y\in\Omega~:~ \psi_{2}(y)<0\}$, $\{y\in\Omega~:~
\psi_{1}(y)<0\mbox{ and }\psi_{2}(y)<0\}$, or $\{y\in\Omega~:~
\psi_{1}(y)<0\mbox{ or }\psi_{2}(y)<0\}$, where $\psi_{j}$
denotes the components of $\psi$.
\end{enumerate}

The boundary $\Gamma$ of $\Omega$ is a piecewise smooth curve. We
work under the assumption that the boundary $\Gamma$  consists of a
finite number of smooth curves $\overline \Gamma_{k}$ for
$k=1,\ldots,m$. The family $(\Gamma_k)$ is the minimal family of
curves making up the boundary $\Gamma$. The curve $\overline
\Gamma_{k+1}$ follows $\overline\Gamma_{k}$ according to a positive
orientation, on each connected component of $\Gamma$. Let
${\mathsf{s}}_{k}$ denotes the vertex which is the end point of
$\overline \Gamma_{k}$. In a neighborhood of $\partial\Omega$,
define a vector field $\nu_{k}$. For each $k$, the vector $\nu_k$ is
 the unit normal a.e. on $\Gamma_{k}$.

Let $\Sigma$ be the set of vertices the domain $\Omega$ has. Suppose
that $\Sigma\not=\emptyset$ and consists exactly of $N$ vertices.
This assumption certifies that the domain $\Omega$ does not have a
smooth boundary but only a {\it piecewise} smooth boundary. When
$\Sigma=\emptyset$, the boundary of $\Omega$ is smooth and the
operator in $L^2(\Omega)$ is the one studied in \cite{Fr, FK}.

At each vertex $\mathsf s_k\in\Sigma$, let $\alpha_{\mathsf{s_k}}$
denotes the angle between $\overline\Gamma_{k}$ and
$\overline\Gamma_{k+1}$ measured towards the interior.

\subsection{Main result}

For each angle $\alpha$, recall the sectorial domain $\Omega_\alpha$
introduced in \eqref{Omega-alpha}. Let $P_{\mathfrak
b,\Omega_\alpha}=-(\nabla-i\mathfrak b\Ab_0)^2$ be the operator in
$L^2(\Omega_\alpha)$ introduced in Section~\ref{sec:alpha} (with
Neumann boundary condition on the smooth part of the boundary of the
sector $\Omega_\alpha$). The number
\begin{equation}\label{eq-nb-alpha'}
\mathsf n(\alpha,\lambda;\mathfrak b)={\rm
tr}\Big(\mathbf
1_{(-\infty,\lambda)}(P_{\Omega_\alpha})\Big)=\sum_{\Lambda<\lambda}{\rm
dim}\,\Big({\rm Ker}\,(P_{\Omega_\alpha}-\Lambda{\rm Id})\Big)
\end{equation}
is finite for all $\alpha$ and $\lambda<\Theta_0$.

If $\mathfrak b=1$, we write,
\begin{equation}\label{eq-nb-alpha}
\mathsf n(\alpha,\lambda)=\mathsf n(\alpha,\lambda;\mathfrak b=1)\,.
\end{equation}
By a scaling argument, it is easy to see that,
$$\mathsf n(\alpha,\lambda;\mathfrak b)=\mathsf n(\alpha,\lambda/\mathfrak b)\,.$$
Let the domain $\Omega$ be as described in Section~\ref{sec:omega}.
Define the magnetic Schr\"odinger operator,
\begin{equation}\label{eq-op:h}
P_{h,\Omega}=-(h\nabla-i\Ab)^2\,,\quad {\rm in}\quad
L^2(\Omega)\,,\end{equation} where $\Ab\in C^2(\Omega;\R^2)$ is the magnetic
potential, $h>0$ is the {\it semi-classical} parameter and
$\Bb=\curl\Ab$ is the magnetic field. The domain of the operator
$P_{h,\Omega}$ is,
\begin{multline*}
D(P_{h,\Omega})=\{u\in
L^2(\Omega)~:~(h\nabla-i\Ab)^k\in
L^2(\Omega)\,,~k=1,2,\\
\nu_j\cdot(h\nabla-i\Ab)u=0{\rm ~on~
}\Gamma_k\,,~k=1,\cdots,m\}\,.
\end{multline*}

Define two constants,
\begin{equation}\label{eq-b:b'}
b=\inf_{x\in\overline\Omega}\Bb(x)\,,\quad b'=\inf_{x\in\partial\Omega}\Bb(x)\,.
\end{equation}

The main result of this section is:

\begin{thm}\label{thm-nb}
Let the constant $\Theta_0$ be as defined in \eqref{Th-gam},
$\lambda\in\big(-\infty,\min(\Theta_0b',b)\big)$ and $N(\lambda h)$
the number of eigenvalues of the operator $P_{h,\Omega}$ below
$\lambda h$ counting multiplicities. Suppose that the magnetic field
$\Bb(x)$ is selected such that $b>0$. There exists a positive number
$h_0$ such that, for all $h\in(0,h_0)$, the following equality holds
true,
\begin{equation}\label{eq-thm-nb}
N(\lambda h)=\sum_{k=1}^N \mathsf n\left(\alpha_{\mathsf s_k},\frac{\lambda}{\Bb(\mathsf s_k)}\right)\,.
\end{equation}
The number $h_0$ depends only on the angles $\alpha_{\mathsf s_k}$ and the domain $\Omega$.
\end{thm}

Theorem~\ref{thm-nb} gives the exact number of low-lying eigenvalues
of $P_{h,\Omega}$ corresponding to corners in the domain $\Omega$.
When $\lambda>\Theta_0$, we will see in Section~\ref{sec:extension}
that corners  no more affect the leading order term of $N(\lambda
h)$.

\subsection{Proof of Theorem~\ref{thm-nb}}\label{sec:nb}

\begin{proof}[Upper bound]
Let $\rho$ be a positive constant satisfying $0<\rho<1$. We start by
choosing a partition of unity $\chi_{k,h}$ introduced in
\cite[Prop.~11.2]{Bon} satisfying:
$$\sum_k|\chi_{k,h}|^2=1\,,~
\sum_{k}|\nabla\chi_{k,h}|^2\leq Ch^{-2\rho}{\rm ~in~}\R^2\,,\quad
{\rm supp}\,\chi_{k,h}\subset B(z_j,c_kh^\rho)\,,$$ with the choice of
indices such that
$$\left\{
\begin{array}{l} z_k=\mathsf s_k~{\rm and~}c_k=1~{\rm for~ all~}
k=1,2,\cdots,N\,,\\
{\rm If~}k\not\in\{1,\cdots,m\}~{\rm and~}z_k\not\in\partial\Omega,~{\rm then~}B(z_k,ckh^\rho)\cap\partial\Omega=\emptyset~{\rm and~}c_k=\frac12\min(\tan\alpha_{\mathsf s_k},1)\,,\\
{\rm If~}k\not\in\{1,\cdots,m\}~{\rm and~}z_k\in\partial\Omega,~{\rm then~}B(z_k,c_kh^\rho)\cap\Sigma=\emptyset\,,
\\\hskip3cm B(z_k,c_kh^\rho)\cap\partial\Omega{\rm~is~connected}\break
~{\rm and~}c_k=\frac12\min(|\tan\alpha_{\mathsf s_k}|,1)\,.
\end{array}\right.$$
Recall that $\Sigma$ is the set of vertices of the domain $\Omega$.
If $u$ is a function in the form domain of $P_{h,\Omega}$, define,
$$q(u)=\int_\Omega|(h\nabla-i\Ab)u|^2\,dx\,.$$
The following decomposition formula holds true for every function
$u$ in the form domain of the operator $P_{h,\Omega}$,
$$q(u)=\sum_{k=1}^mq(\chi_{k,h}u)+\sum_{\substack{k\geq
m\\z_k\in\partial\Omega}}q(\chi_{k,h}u)+\sum_{\substack{j\geq
m\\z_k\not\in\partial\Omega}}q(\chi_{k,h}u)-h^2\sum_k\int_\Omega|\nabla\chi_{k,h}|^2|u|^2\,dx\,.$$
Using the upper bound on $|\nabla\chi_{k,h}|^2$, we get the lower
bound,
\begin{equation}\label{eq-q(u):geq}
q(u)\geq\sum_{k=1}^mq(\chi_{k,h}u)+\sum_{\substack{k\geq
m\\z_k\in\partial\Omega}}q(\chi_{k,h}u)+\sum_{\substack{k\geq
m\\z_k\not\in\partial\Omega}}q(\chi_{k,h}u)-Ch^{2-2\rho}\int_\Omega|u|^2\,dx\,.\end{equation}
Let $k\in\{1,2,\cdots,m\}$. It is proved in \cite[p.~252]{Bon} that
by performing a change of variable $y=\psi_k(x)$ and a gauge
transformation (defined by a function $\phi_k$), the following lower
bound holds true,
\begin{equation}\label{eq-q(u):cor'}
q(\chi_{k,h}u)\geq \int_{\Omega_{\alpha_{\mathsf s_k}}}\Big((1-Ch^\rho-Ch^{2\theta})
|(h\nabla-iB_k\Ab_0)v_k|^2-Ch^{4\rho-2\theta}|v_k|^2\Big)\,dy\,.\end{equation}
Here $\theta\in(0,1)$ is any constant, $B_k=\Bb(z_k)$, $\Ab_0$ is
the magnetic potential in \eqref{eq:mp} and
$$v_k(y)=e^{-i\phi_k(y)}\,\Big(\chi_{k,h}\,u\Big)\circ\psi_k^{-1}(y)\,.$$
The optimal choice of $\rho$ and $\theta$ is $\rho=3/8$ and
$\theta=1/8$. This produces the lower bound:
\begin{equation}\label{eq-q(u):cor}
q(\chi_{k,h}u)\geq \int_{\Omega_{\alpha_{\mathsf s_k}}}\Big((1-Ch^{1/4})|(h\nabla-iB_k\Ab_0)v_k|^2-Ch^{5/4}|v_k|^2\Big)\,dy\,.\end{equation}
Similarly, when $k\not\in\{1,2,\cdots,m\}$ and
$z_k\in\partial\Omega$, by applying a change of variable and a gauge
transformation, we obtain the lower bound (see \cite{Fr}),
\begin{equation}\label{eq-q(u):bnd}
q(\chi_{k,h}u)\geq \int_{U_h}\Big((1-Ch^{1/4})|(h\nabla-iB_k\Ab_0)v_k|^2-Ch^{5/4}|v_k|^2\Big)\,dy\,,\end{equation}
where $U_h=\{y=(y_1,y_2)\in\R^2~:~y_1\in(0,h^{3/8})~{\rm
and~}y_2\in(0,\infty)\,\}$. By inserting the lower bounds in
\eqref{eq-q(u):cor} and \eqref{eq-q(u):bnd} into \eqref{eq-q(u):geq}
we obtain,
\begin{align}\label{eq-q(u):geq'}
q(u)\geq& \sum_{k=1}^m\int_{\Omega_{\alpha_{\mathsf s_k}}}\Big((1-Ch^{1/4})|(h\nabla-iB_k\Ab_0)v_k|^2-Ch^{5/4}|v_k|^2\Big)\,dy\\
&+\sum_{\substack{k\geq m\\z_k\in\partial\Omega}}\int_{U_h}\Big((1-Ch^{1/4})|(h\nabla-iB_k\Ab_0)v_k|^2-Ch^{5/4}|v_k|^2\Big)\,dy\nonumber\\
&+\sum_{\substack{k\geq
m\\z_k\not\in\partial\Omega}}q(\chi_{k,h}u)-Ch^{5/4}\int_\Omega|u|^2\,dx\,.\nonumber\end{align}
By the variational min-max principle, we deduce the following upper
bound on the number $N(\lambda h)$,
\begin{multline*}N(\lambda
h)\leq \sum_{k=1}^m \mathsf n\Big(\alpha_{\mathsf
s_k},\frac{\lambda+Ch^{1/4}}{(1-Ch^{1/4})\Bb(\mathsf s_k)}\Big)+\\
\sum_{\substack{k\geq m\\z_k\in\partial\Omega}}
N\Big(P_{B_k,U_h},\frac{\lambda h+Ch^{5/4}}{1-Ch^{1/4}}\Big)+
\sum_{\substack{k\geq m\\z_k\not\in\partial\Omega}}
N\Big(P_{h,\Omega}^D,\lambda h+Ch^{5/4}\Big)\,.\end{multline*} Here
$P_{B_k,U_h}=-(h\nabla-iB_j\Ab_0)^2$ is the operator in $L^2(U_h)$
with Neumann boundary condition at $y_2=0$ and Dirichlet condition
elsewhere, and $P_{h,\Omega}^D=-(h\nabla-i\Ab)^2$ is the operator in
$L^2(\Omega)$ with Dirichlet boundary condition. The spectrum of
$P_{B_k,U_h}$ starts at $\Theta_0B_kh$ and that of $P_{h,\Omega}^D$
starts above $bh$ (cf. \cite{Fr}). Thus, if
$\lambda<\min(\Theta_0b',b)$ and $h$ is selected sufficiently small,
we get,
$$N\Big(P_{B_k,U_h},\frac{\lambda h+Ch^{5/4}}{1-Ch^{1/4}}\Big)=N\Big(P_{h,\Omega}^D,\lambda
h+Ch^{5/4}\Big)=0\,.$$ Since the spectrum of the operator
$P_{\Omega_{\alpha_{\mathsf s_k}}}$ below $\Theta_0$ consists  of isolated eigenvalues, we get for $h$ sufficiently
small that,
$$\mathsf n\Big(\alpha_{\mathsf
s_k},\frac{\lambda+Ch^{1/4}}{1-Ch^{1/4}}\Big)=\mathsf n\Big(\alpha_{\mathsf
s_k},\frac{\lambda}{\Bb(\mathsf s_k)}\Big)\,.$$ This finishes the proof of the upper bound.
\end{proof}

\begin{proof}[Lower bound]
The proof of the lower bound is similar to the one in \cite{Fr} and uses a bracketing technique. Let $\widetilde P_{h,\Omega}$ be the self adjoint realization of the restriction of the operator $P_{h,\Omega}$ on functions  vanishing outside the set,
$$\bigcup_{k=1}^N B(\mathsf s_k,h^{3/8})\,.$$
By the variational min-max principle, the eigenvalues of $\widetilde
P_{h,\Omega}$ are larger than those of $P_{h,\Omega}$. Thus,
$$N(\lambda h)\geq N(\lambda h,\,\widetilde P_{h,\Omega})\,.$$
We will show next that there exist constants $C>0$ and $h_0>0$ such that, for all $h\in(0,h_0]$, we have,
\begin{equation}\label{eq-lb-nb}
N(\lambda h,\,\widetilde P_{h,\Omega})\geq \sum_{k=1}^N
\mathsf n\Big(\alpha_{\mathsf s_k},\frac{\lambda-Ch^{1/4}}{(1+Ch^{1/4})\Bb(\mathsf s_k)}\Big)\,.\end{equation}
If $h$ is made sufficiently small, then
$$ \mathsf n\Big(\alpha_{\mathsf s_k},\frac{\lambda-Ch^{1/4}}{(1+Ch^{1/4})\Bb(\mathsf s_k)}\Big)=
\mathsf n\Big(\alpha_{\mathsf s_k},\frac{\lambda}{\Bb(\mathsf s_k)}\Big)\,,$$
and we get the lower bound we wish to prove.

Derivation of \eqref{eq-lb-nb} is easy. Let $\widetilde q$ be the quadratic form associated to
$\widetilde P_{h,\Omega}$.

Select an arbitrary function
$u$  in the form domain of
$\widetilde P_{h,\Omega}$.  A matching asymptotic upper bound to \eqref{eq-q(u):cor} is proved in \cite[p. 252]{Bon}:
$$\int_{B(\mathsf s_k,h^{3/8})}|(\nabla-i\Ab)u|^2\,dx\leq
\int_{\Omega_{\alpha_{\mathsf s_k}}}\Big((1-Ch^{1/4})|(h\nabla-iB_k\Ab_0)v_j|^2-Ch^{5/4}|v_k|^2\Big)\,dy\,,$$
where $\Bb_k=\Bb(\mathsf s_k)$ and $v_k$ is obtained from $u$ by the change of variables
$y=\psi_k(x)$ times a gauge transformation.
Summing over $k$, we get,
$$\widetilde q(u)\geq \int_{\Omega_{\alpha_{\mathsf s_k}}}\Big((1-Ch^{1/4})|(h\nabla-iB_k\Ab_0)v_k|^2-Ch^{5/4}|v_k|^2\Big)\,dy\,.$$
By the variational min-max principle, we deduce that the eigenvalues of the operator $\widetilde P_{h,\Omega}$ are
larger than those of the direct sum of the operators $P_{\Omega_{\alpha_{\mathsf s_k}},\Bb_k}$,
thereby proving \eqref{eq-lb-nb}.
\end{proof}

\section{Energy and number of eigenvalues in piecewise smooth
domains}\label{sec:extension}

\subsection{Main results}
We will state other results concerning the number and sum of
eigenvalues of the operator $P_{h,\Omega}$ introduced in
\eqref{eq-op:h}. The assumptions on the domain $\Omega$  is as
described in Section~\ref{sec:omega}. The notation and assumption on
the magnetic field $\Bb=\curl\Ab$ is as in \eqref{eq-b:b'} and
Section~\ref{sec:nb}.

Let $(e_j)$ be the increasing sequence of  eigenvalues
of the operator $P_{h,\Omega}$ in the interval $(-\infty,bh]$, counting multiplicities.
If $\lambda\in(-\infty, b]$, define,
\begin{align}
&N(\lambda h)={\rm tr}\Big(\mathbf 1_{(-\infty, \lambda h)}(P_{h,\Omega})\Big)=
{\rm Card}\{e_j~:~e_j\leq \lambda h\}\,,\label{eq-nb-sc}\\
& E(\lambda h)={\rm tr}\Big(\mathbf 1_{(-\infty, \lambda h)}(P_{h,\Omega}-\lambda h{\rm Id})\Big)=\sum_{e_j\leq\lambda h}\big(e_j-\lambda h\big)\,.\label{eq-en-sc}
\end{align}

The main result we prove in this section is Theorem~\ref{thm:FK;Fr}.
Its statement requires the notation $[x]_+=\max(x,0)$ for any real
number $x$, the eigenvalue $\mu_1(\xi)$ introduced in
\eqref{l1-ga,xi}, and the arc-length measure $ds(x)$ along the
boundary of $\Omega$.

\begin{thm}\label{thm:FK;Fr}
For any real number $\lambda\leq b$ and as $h\to0$, it holds true that,
\begin{equation}\label{eq-en:asy}
E(\lambda h)=-\frac{h^{1/2}}{2\pi}
\int_{\partial\Omega}\int_{-\infty}^\infty \Bb(x)^{3/2}\left[\frac{\lambda}{\Bb(x)}-\mu_1(\xi)\right]_+\,d\xi\,ds(x)
+o(h^{1/2})\,.
\end{equation}
Furthermore, if $\lambda<b$, then,
\begin{equation}\label{eq-nb:asy}
N(\lambda h)=\frac{1}{2\pi h^{1/2}}
\int_{\{(x,\xi)\in\partial\Omega\times \R~:~\Bb(x)\mu_1(\xi)\leq \lambda}
\Bb(x)^{1/2}\,d\xi\,ds(x)+o\left(h^{-1/2}\right)\,.
\end{equation}
\end{thm}

While proving Theorem~\ref{thm:FK;Fr}, we only give the new
ingredients and constructions required to adapt the proof given in
\cite{FK}. We will  refer to \cite{FK} for the detailed
calculations.

\subsection{Proof of Theorem~\ref{thm:FK;Fr}}

Notice that \eqref{eq-nb:asy} is a consequence of \eqref{eq-en:asy}.
Actually, the term on the right side of \eqref{eq-nb:asy} is the
derivative with respect to $\lambda h$ of that on the right side of
\eqref{eq-en:asy}. On the other hand, using the variational
principle of Lemma~\ref{lem-VP2}, the number $N(\lambda h)$ can be
seen as the derivative of the energy $E(\lambda h)$ with respect to
$\lambda h$.

Therefore, to prove Theorem~\ref{thm:FK;Fr}, it is sufficient to establish the asymptotic
formula in \eqref{eq-en:asy}.

Using a magnetic Lieb-Thirring inequality and a decomposition of the operator by a partition of unity, it is possible to prove
that the energy $E(\lambda h)$ is finite for all $\lambda\leq b$. Details regarding this proof
are given in \cite[Section~5.1, pp. 242-243]{FK}.

We start by proving the asymptotic lower bound:

\begin{equation}\label{eq-en:asy-lb}
E(\lambda h)\geq-\frac{h^{1/2}}{2\pi}
\int_{\partial\Omega}\int_{-\infty}^\infty \Bb(x)^{3/2}\left[\frac{\lambda}{\Bb(x)}-\mu_1(\xi)\right]_+\,d\xi\,ds(x)+o(h^{1/2})\,.
\end{equation}
Let $H=P_{h,\Omega}-\lambda h{\rm Id}$ and notice that the energy $E(\lambda h)$ can be expressed in the more useful
form,
$$E(\lambda h)=\tr\left( H \mathbf 1_{(-\infty,0)}(H)\right)\,.$$
Let $\{f_1,f_2,\ldots,f_N\}$  be any $L^2$ orthonormal set in
$D(H)$. We will give a uniform lower bound to
$$\sum_{j=1}^N\langle f_j\,,\,Hf_j\rangle\,.$$
Using Lemma~\ref{lem-VP2}, this will imply a lower bound to ${\rm
tr}
\left( H \mathbf 1_{(-\infty,0)}(H)\right)$.%

Consider a partition of unity of $\R$,
\begin{equation}\label{partition-R}
\psi_1^2+\psi_2^2=1,\quad\supp\,\psi_1\subset]-\infty,1[,\quad
\supp\,\psi_2\subset[\frac12,\infty[\,,\end{equation}
such that,
\begin{equation}\label{partition-R'}
|\nabla\psi_1|^2+|\nabla\psi_2|^2\leq C,\end{equation}
and $C>0$ is a universal constant.

Let
\begin{equation}\label{eq-tau}
\tau(h)=\frac12\min({|\tan\alpha|},1)\,h^{3/8}\,,\end{equation}
where $\alpha=\min\{\alpha_{\mathsf s_k}~:~k=1,2,\cdots, m\}$.

Using
the partition of unity in (\ref{partition-R}), we put
\begin{equation}\label{eq-pu:asy}
t(x)={\rm dist}(x,\partial\Omega)\,,\quad \psi_{1,h}(x)=\psi_1\left(\frac{t(x)}{\tau(h)}\right),\quad
\psi_{2,h}(x)=\psi_2\left(\frac{t(x)}{\tau(h)}\right),
\quad\forall~x\in\overline\Omega.\end{equation}
We introduce the potential
\begin{equation}\label{potential-2}
V_{h}=h^2\left(|\nabla\psi_{1,h}|^2+|\nabla\psi_{2,h}|^2\right)\,.
\end{equation}
It is possible to prove that (see \cite[Proof of (5.10), pp. 243]{FK}):
\begin{align}\label{int-estimate}
\sum_{j=1}^N\langle f_j\,,\,Hf_j\rangle\geq \sum_{j=1}^N\langle
f_j\,,\,\psi_{1,h}
(H-V_h)\psi_{1,h}f_j\rangle
-C\frac{h}{\tau(h)}\left(1+\frac{h}{\tau(h)^2}\right)\,.
\end{align}
With our choice of $\tau(h)$, the remainder terms in \eqref{int-estimate} are of the order of $h^{5/8}=o(h^{1/2})$.
Thus, we have,
\begin{align}\label{int-estimate'}
\sum_{j=1}^N\langle f_j\,,\,Hf_j\rangle\geq \sum_{j=1}^N\langle
f_j\,,\,\psi_{1,h}
(H-V_h)\psi_{1,h}f_j\rangle
+o(h^{1/2})\,.
\end{align}
Next, we estimate from below the term on the right side of \eqref{int-estimate'}. Recall the vertices $\mathsf s_k$, $k=1,2,\cdots,m$,
 of the domain $\Omega$. Recall also that the boundary of the domain $\Omega$ consists of smooth curves $(\Gamma_k)$.
 For each $k$, define,
\begin{equation}\label{eq-nb:asy;gam}
\Gamma_k(h)=\{x\in\overline\Omega~:~{\rm dist}(x,\Gamma_k)\leq \tau(h)\quad{\rm and}~{\rm dist}(x,s_k)\geq \frac12h^{3/8}\}\,,
\end{equation}
 where $\tau(h)$ is defined in \eqref{eq-tau}.
  Consider a partition of unity of $\R^2$\,:
 $$\sum_{k=1}^m g_k^2+\sum_{k=1}^m h_k^2=1\,,\quad \sum_{k=1}^m\left(|\nabla g_k|^2+|\nabla h_k|^2\right)\leq \frac{C}{h^{3/4}}\,,$$
where  ${\rm supp}\,g_k\subset B(\mathsf s_k,h^{3/8})$, ${\rm supp}\,h_k\subset \Gamma_k(h)$ for all $k$, and
$C$ is a universal constant.

Then, we have the decomposition formula,
\begin{multline}\label{eq-nb-lb:asy}
\sum_{j=1}^N\langle
f_j\,,\,\psi_{1,h}
(H-V_h)\psi_{1,h}f_j\rangle=\sum_{j=1}^N\sum_{k=1}^m\langle
f_j\,,\,g_k\psi_{1,h}
(H-V_{1,h})g_k\psi_{1,h}f_j\rangle\\+
\sum_{j=1}^N\sum_{k=1}^m\langle
f_j\,,\,h_k\psi_{1,h}
(H-V_{1,h})h_k\psi_{1,h}f_j\rangle\,,
\end{multline}
where
$$V_{1,h}=V_h+h^2\sum_{k=1}^m|\nabla g_k|^2+h^2\sum_{k=1}^m|\nabla h_k|^2\,,$$
and
$$|V_{1,h}|\leq Ch^{5/4}\,.$$
We will show that the first term on the right side of \eqref{eq-nb-lb:asy} is of order $o(h^{1/2})$.
Recall the definition of the energy $\mathcal E_{\rm corn}(\alpha,b,R,\lambda)$ in \eqref{eq-sec:en}. We use \eqref{eq-q(u):bnd} with $u=\psi_{1,h}f_j$ and $\chi_{k,h}=g_k$ to bound the term $\langle
f_j\,,\,g_k\psi_{1,h}
(H-V_{1,h})g_k\psi_{1,h}f_j\rangle$ from below by a quadratic form defined over the sector $\Omega_{\alpha_{\mathsf s_j}}$. Then we
use the variational principle of Lemma~\ref{lem-VP2} and the lower bound of $V_{1,h}$ to get,
$$
\sum_{j=1}^N\langle
f_j\,,\,g_k\psi_{1,h}
(H-V_{1,h})g_k\psi_{1,h}f_j\rangle\geq
-(1-Ch^{1/4})\,\mathcal E_{\rm corn}(\alpha_{\mathsf s_k},\Bb_k,h^{-1/8},\widetilde\lambda)\,.
$$
Here $\widetilde\lambda =\lambda b^{-1}-1+ Ch^{1/4}\leq Ch^{1/4}$. Clearly, $\widetilde\lambda$ can be made smaller than an arbitrary positive number. Thus, we can apply Lemma~\ref{thm-disspec} and use \eqref{eq-sec:en'}. In this way, we get the following lower bound that holds uniformly with respect to $k$,
\begin{equation}\label{eq-nb-er:asy}\sum_{j=1}^N\langle
f_j\,,\,g_k\psi_{1,h}
(H-V_{1,h})g_k\psi_{1,h}f_j\rangle\geq
-Ch^{3/4}=o(h^{1/2}).
\end{equation}

The second term on the right side of \eqref{eq-nb-lb:asy} is bounded as in \cite[Proof of (5.26); pp. 244-248]{FK}.
The following lower bound holds uniformly in $k$, $N$ and the orthonormal  family $\{f_j\}$:
\begin{align}\label{conclusion}
\sum_{j=1}^N\langle
f_j\,,\,h_k\psi_{1,h}
(H-V_{1,h})h_k\psi_{1,h}f_j\rangle
\geq -\frac{h^{1/2}}{2\pi}\int_{\Gamma_k} \int_{\mathbb R}
B(x)^{3/2}\left[\mu_1(\xi)-\frac{\lambda}{B(x)}\right]_- \, d  \xi d
x-Ch^{5/8}\,.\end{align}
Substitution of \eqref{eq-nb-er:asy} and \eqref{conclusion} into \eqref{eq-nb-lb:asy} establishes the lower bound \eqref{eq-en:asy-lb}.

\subsubsection*{Upper bound}\label{sec-thm1-ub}

The upper bound will be  obtained by constructing a specific density matrix $\gamma$ and computing the energy of
${\rm tr}\Big((P_{h,\Omega}-\lambda{\rm Id})\gamma\Big)$. The calculations follow closely those in \cite[Section~5.3, pp. 248]{FK}.

Recall the sets $\Gamma_k(h)$ introduced in \eqref{eq-nb:asy;gam}.  Let $\tau(h)$ be as defined in \eqref{eq-tau}. For each $k$, we cover $\Gamma_k(h)$ by disjoint squares,
$$K_{j,k}=\{\,x\in\R^2~:~|x-z_{j,k}|< \tau\,\}\,,\quad 1\leq j\leq N_k\,,$$
where the points $(z_{j,k})$ are on $\Gamma_k$ and equally spaced. In each $K_{j,k}$, it is possible to apply a transformation $\Phi_{j,k}(x)=(s,t)$ that straightens $\partial\Omega$ ($t=0$ defines the boundary of $\Omega$ and $s$ measures the arc-length along $\partial\Omega$).

Recall the eigenvalues $(\mu_p(\xi)$ of the harmonic oscillator in \eqref{L-dop}. We introduce an orthonorml basis $(u_p)$ of $L^2(\R_+)$\,:
\begin{equation}\label{eq-ef-ho}
\left\{\begin{array}{l}
\Big(-\partial_t^2+(t-\xi)^2\Big)u_p\left( t;
\xi\right)=\mu_p(\xi)\,u_p\left( t;
\xi\right)\,,\quad t>0\,,\\
u_p'\left( 0;
\xi\right)=0\,,\\
\displaystyle\int_0^\infty|u_p\left( t;
\xi\right)|^2\,dt=1\,.
\end{array}\right.
\end{equation}

Let  $\chi \in
C^{\infty}(\R^2)$ be  positive, smooth, have values in $[0,1]$, supported in $B(0,2)$ and equal to $1$ on $B(0,1)$.
Define $\chi_{j,k}(x)$  to be
\begin{align*}
\chi_{j,k}(x):=\chi\Big(\frac{x-z_{j,k}}{\tau(h)}\Big)\,.
\end{align*}
We denote by $B_{j,k}=\Bb(z_{j,k})$. For each $(j,k)$, there exists
a function $\phi_{j,k}$ such that
$$\Ab-\nabla \phi_{j,k}=\Big(-{B_{j,k}}\,x_2,0\Big)+\mathcal O\Big(|x-z_{j,k}|\Big)\quad {\rm in~} K_{j,k}\,.$$
Define
\begin{align}
\widetilde{f}_{j,p,k}( (s,t); h,, \xi) :=
\left(\frac{B_{j,k}}{h}\right)^{1/4} e^{-i\xi s
\sqrt{B_{j,k}/h}}\,u_p\left(  \sqrt{\frac{B_{j,k}}{h}}\,t;
\xi\right) \,.
\end{align}
We get---by the coordinate transformation $\Phi_{j,k}$---the following
function in $\Omega$,
\begin{align}
f_{j,p,k}(x; h, \xi) = \widetilde{f}_{j,p,k}( \Phi_{j,k}(x); h, \xi)\,\chi_{j,k}(x)\psi_{1,h}(x)\,.
\end{align}
Next, let $K>0$ and define the function,
$$f_{p}(x\,;\xi)=\sum_{k=1}^m\sum_{j=1}^{N_k}\,M_p(h, \xi,j,k\,;K)\,e^{i\phi_{j,k}(x)/h}\chi_{j,k}(x)\,\psi_{1,h}(x)\,\widetilde{f}_{j,p,k}( \Phi_{j,k}(x); h, \xi)\,,$$
where $M_p(h, \xi,j,k\,;K)=M^K(h,\xi,j,k)\mathbf 1_{\{p=1\}}$ and $M^K(h,\xi,j,k)$  is the characteristic function of the set:
\begin{align*}
\{ (\xi,j,k) \,:\, \frac{\lambda}{B_{j,k}} - \mu_1(\xi) \geq 0,~ |\xi| \leq K \}\,.
\end{align*}
Let $\Pi_p^{\rm bnd}(h,\xi)$ be the operator with integral kernel  in $L^2(\Omega)$ whose integral kernel is defined by,
\begin{align}
\Pi_p^{\rm bnd}(h,\xi)(x,x') =  f_p(x\,;  \xi)\,
\overline{f_p(x'\,;  \xi)}\,.
\end{align}
Define the density matrix,
\begin{align}\label{den-ub}
\gamma = \frac{1}{2 \pi \sqrt{h/B_{\sigma}}}\sum_{p=1}^{\infty} \Pi_p^{\rm
bnd}(h,\xi)\,.
\end{align}
Clearly $0 \leq \gamma$ as an operator on $L^2(\Omega)$. Furthermore, there exists  a constant $C>0$ such that,
\begin{align}
\label{eq:BdI}
\gamma \leq (1 + C\tau(h))\,,
\end{align}
in the quadratic form sense. Details concerning the derivation of \eqref{eq:BdI} are give in \cite[Proof of (5.34)]{FK}. Moreover, (see \cite[(5.37)]{FK}),
\begin{align}\label{eq-nb:asy;conc}
\tr[ (P_{h,\Omega} - \lambda h) \gamma] &
\leq - \frac{h^{1/2}}{2 \pi} \int_{-K}^K \int_{\partial\Omega} \Bb(x)^{3/2} \big[ \frac{\lambda}{\Bb(x)} - \mu_1(\xi) \big]_{+} ds(x) d\xi  + C K h^{3/4}\,.
\end{align}
Let us mention that, while estimating the term on the left side of \eqref{eq-nb:asy;conc}, the discrete sum over $j$ and $k$ becomes a Riemann sum as $h\to0$, thereby resulting in an integral on the
right side of \eqref{eq-nb:asy;conc}.
Since $K$ can be chosen arbitrarily large, \eqref{eq-nb:asy;conc} implies the asymptotic upper bound,
$$\tr[ (P_{h,\Omega} - \lambda h) \gamma]
\leq - \frac{h^{1/2}}{2 \pi} \int_{-\infty}^\infty \int_{\partial\Omega} \Bb(x)^{3/2} \big[ \frac{\lambda}{\Bb(x)} - \mu_1(\xi) \big]_{+} ds(x) d\xi  + o(h^{1/2})\,.
$$

\section*{Acknowledgments}
The authors are supported by a grant from Lebanese University.

\end{document}